\documentclass[10pt]{amsart}

\usepackage{amssymb,amsmath} %some useful extra symbol packages
\usepackage{multirow} %for merging rows of matrices
\usepackage{enumitem} %to change the bullet/character that \item uses
\usepackage{graphicx} %to rotate symbols
\usepackage{calc} %for setting the width of columns
\usepackage{listings}
\usepackage{verbatim}
\usepackage{arydshln}
\usepackage{nicefrac}
\usepackage{setspace}

\newcommand{\lline}[1]{\multicolumn{1}{:c}{#1}}
\def\rddots{\cdot^{\cdot^{\cdot}}}   
   \newcommand{\F}{\mathcal{F}}
\newtheorem{proposition}{Proposition}[section]             % numbering schemes for various environments

\newtheorem{theorem}[proposition]{Theorem}
\newtheorem{corollary}[proposition]{Corollary}
\newtheorem{remark}[proposition]{Remark}

\newtheorem{example}[proposition]{Example}

\begin{document}
\title{Bounds on polynomial roots\\ using intercyclic companion matrices\footnote{\copyright\ 2017. T\MakeLowercase{his manuscript version is made available under the} CC-BY-NC-ND 4.0 
		\MakeLowercase{license \newline http://creativecommons.org/licenses/by-nc-nd/4.0/ }}}
\date{Preprint. November 1, 2017.} %submission date

\author{Kevin N. Vander Meulen, Trevor Vanderwoerd}
\address{Department of Mathematics\\
Redeemer University College, Ancaster, ON, L9K 1J4, Canada.
\let\thefootnote\relax\footnote{Current address (T. Vanderwoerd): Department of Civil and Environmental Engineering, University of Waterloo, ON, N2L 3G1,
	Canada.}}
\email{kvanderm@redeemer.ca, tvanderwoerd@redeemer.ca}

\keywords{roots of polynomials, bounds, eigenvalues, Fiedler companion matrix, sparse companion matrix}
\subjclass[2010]{15A18, 15A42, 15B99, 26C10, 65F15, 65H04}

\begin{abstract}
{The Frobenius companion matrix, and more recently the Fiedler companion matrices, have been used to
provide lower and upper bounds on the modulus of any root of a polynomial $p(x)$. 
In this paper we explore
new bounds obtained from taking the $1$-norm and $\infty$-norm of a matrix in 
the wider class of intercyclic companion matrices. 
As is the case with Fiedler matrices,
we observe that the new bounds from intercyclic companion matrices can improve those from
the  Frobenius matrix by at most a factor of two.    
By using the Hessenberg  
form of an intercyclic companion matrix, we describe how to
determine the best upper bound when restricted to Fiedler companion matrices using the $\infty$-norm. We also obtain
a new general bound by considering the polynomial $x^qp(x)$ for $q>0$.
We end by considering upper bounds obtained from inverses of monic reversal polynomials of intercyclic companion
matrices, noting that these can make more significant improvements on the bounds from a 
Frobenius companion matrix for certain polynomials.
}
\end{abstract}

\maketitle

\section{Introduction}
There are various techniques for approximating the roots of a polynomial (see for example \cite{HJ, Marden}).
Some algorithms for determining the roots of a polynomial rely on a good first approximation 
(see e.g. \cite{BF}) based on the coefficients of the polynomial.
One method \cite{Melman} for finding the roots of a polynomial $p(x)$ is to find the eigenvalues
of a companion matrix, since a companion matrix has characteristic polynomial $p(x)$.
To approximate the roots of $p(x)$, one can apply
Gershgorin's Theorem \cite{Melman} or use matrix norms \cite{DDP} on a companion matrix to find
regions in the complex plane to locate the eigenvalues. For example, using these methods, 
one can obtain Cauchy's bound: if $\lambda$ is a root of 
\begin{equation}\label{poly}
p(x) = x^n + a_{n-1}x^{n-1} + a_{n-2}x^{n-2} + \cdots + a_1x + a_0
\end{equation}
then 
\begin{equation}\label{Cauchy}
\vert\lambda \vert \leq \max\{\vert a_0 \vert,1+\vert a_1\vert,1+\vert a_2\vert,\ldots,1+\vert a_{n-1}\vert\}.
\end{equation}
Typically one uses the classical \emph{Frobenius companion matrix}, 
\begin{equation}\label{frob}
\F = \begin{bmatrix}
0      & 1      & 0      & \cdots & 0        \\
0      & 0      & 1      & \cdots & 0        \\
\vdots & \vdots & \vdots & \ddots & \vdots   \\
0      & 0      & 0      & \cdots & 1        \\
-a_0   & -a_1   & -a_2   & \cdots & -a_{n-1} \\
\end{bmatrix},
\end{equation}
for these purposes but 
recently other companion matrices have been discovered. In particular, the Fiedler
companion matrices, introduced in ~\cite{Fiedler}, were explored in \cite{DDP}
to provide upper and lower bounds on the modulus of a root of $p(x)$. 
The Frobenius matrix is itself a Fiedler matrix.
More recently, 
the sparse companion matrices (also known as intercyclic companion matrices) 
were characterized in \cite{EKSV}. This class of matrices
includes the Fiedler matrices as a special case. In this paper we develop
new bounds on the modulus of a root of $p(x)$ using the larger class of sparse 
companion matrices, comparing them with other known bounds, especially those in \cite{DDP}.
We also show how some of the previous bounds are more easily obtained  by using
the Hessenberg structure of sparse companion matrices.

We start by providing formal definitions of our terms and describing
the Fiedler companion matrix in their Hessenberg form in Section~\ref{secIntro}.
In Section~\ref{secStraightNonFied}, using the $1$-norm and $\infty$-norm
 we develop upper bounds based on sparse companion matrices,
extending results that are known~\cite{DDP} for the smaller class of Fiedler matrices.
In Section~\ref{secFLshape} we use the
Hessenberg structure of sparse companion matrices to help choose the 
Fiedler companion matrix that provides the best upper bound on the modulus
of a root of $p(x)$. In Section~\ref{secExtendedPoly} we discuss the usefulness of
applying the techniques to a polynomial $x^qp(x)$ with $q>0$ to obtain bounds on the
roots of $p(x)$. Using monic reversal polynomials, we develop lower bounds on
the modulus of a root of $p(x)$ in Section~\ref{secReversal}. Then in 
Section~\ref{secInverseReversals} we consider lower and upper bounds by 
using the inverse of a sparse companion matrix.

\section{Formal Definitions}\label{secIntro}
%%%%%%%%%%%%%%%%%%%%%%%%%%%%%%%%%%%%%%%%%%%%%%%%
Formally, as in \cite{EKSV}, we say a \emph{companion matrix} 
to $p(x)=x^n + a_{n-1}x^{n-1} + \cdots + a_0$ is an $n\times n$ matrix $C=C(p)$ over a field $\mathbb{F}[a_0 ,a_1, \ldots,a_{n-1}]$ with $n^2 -n$ entries constant in $\mathbb{F}$ and the remaining entries variables $-a_0 ,-a_1 ,\ldots,-a_{n-1}$ such that the characteristic polynomial of $C$ is $p$. For example, Figures~\ref{frcompanion} and \ref{fFiedler} 
display companion matrices with characteristic polynomial $p=x^5+a_4x^4+a_3x^3+a_2x^2+a_1x+a_0,$ 
and Figure~\ref{fsparse} displays companion matrices of order $6$ with characteristic polynomial
$p=x^6+a_5x^5+a_4x^4+a_3x^3+a_2x^2+a_1x+a_0.$
\begin{figure}[ht]
\[ \left[ \begin{array}{ccccc}
0    & 1    & 0    & 0    & 0 \\
0    & 0    & 1    & 0    & 0 \\ 
0    & 0    & 0 & 1    & 0 \\
0    & 0 & 0 & 0    & 1 \\
-a_0 & -a_1 & -a_2    & -a_3    & -a_4 \\
\end{array} \right],
\left[ \begin{array}{rrrrr}
0    & 1    & 0    & -a_3    & -a_2 \\
0    & 0    & -a_0    & 8    & -a_1 \\
0    & 0    & 0    & 0    & 1 \\ 
1    & 0    & -8 & -a_4 & 0 \\
0 & 0 & 0 & 1    & 0 \\
\end{array} \right], 
\left[ \begin{array}{ccccc}
-a_4    & 1    & 0    & 0    & 0 \\ 
-a_3    & 0 & 1    & 0    & 0 \\
-a_2    & 0 & 0    & 1    & 0 \\
-a_1    & 0 & 0    & 0    & 1 \\
-a_0 & 0 & 0    & 0    & 0 \\
\end{array} \right] \]
\caption{Various companion matrices}\label{frcompanion}
\end{figure}

The first matrix in Figure~\ref{frcompanion} is a Frobenius companion matrix.
The Frobenius companion matrix $\F$
is sometimes represented in other Hessenberg forms, for example the third matrix
in Figure~\ref{frcompanion} has also been referred to as a Frobenius companion matrix.
In particular, $\F^T$ has the variables
in the last column, $R\F R$ has the variables in the first row, and $R\F^TR$ has the variables 
in the first column (e.g. the last matrix in Figure~\ref{frcompanion}), using the reverse permutation 
matrix $R=\arraycolsep=0.8pt\def\arraystretch{0.8}{\tiny{\left[\begin{array}{crl}0&&1\\&\rddots&\\1&&0\end{array}\right]}}$. We will say that two companion matrices
are \emph{equivalent} if one can be obtained from the other by permutation similarity and/or transposition.  
As such, each Frobenius companion matrix of order $n$ is equivalent to $\F$ in (\ref{frob}).

\begin{figure}[ht]
\[ \left[ \begin{array}{ccc|ccc}
0    & 1    & 0    & 0    & 0    & 0    \\
0    & 0    & 1    & 0    & 0    & 0    \\ \hline
0    & -a_4 & -a_5 & 1    & 0    & 0    \\
0    & -a_3 & 0    & 0    & 1    & 0    \\
-a_1 & 0    & 0    & 0    & 0    & 1    \\
-a_0 & 0    & -a_2 & 0    & 0    & 0    \\
\end{array} \right],
\left[ \begin{array}{cc|cccc}
0    & 1    & 0    & 0    & 0    & 0    \\ \hline
-a_4 & -a_5 & 1    & 0    & 0    & 0    \\
0    & 0    & 0    & 1    & 0    & 0    \\
-a_2 & -a_3 & 0    & 0    & 1    & 0    \\
0    & 0    & 0    & 0    & 0    & 1    \\
-a_0 & -a_1 & 0    & 0    & 0    & 0    \\
\end{array} \right], 
\left[ \begin{array}{cccc|cc}
0    & 1    & 0    & 0    & 0    & 0    \\
0    & 0    & 1    & 0    & 0    & 0    \\
0    & 0    & 0    & 1    & 0    & 0    \\ \hline
-a_2 & 0    & -a_4 & -a_5 & 1    & 0    \\
-a_1 & 0    & 0    & 0    & 0    & 1    \\
-a_0 & 0    & 0    & -a_3 & 0    & 0    \\
\end{array} \right] \]
\caption{Sparse companion matrices}\label{fsparse}
\end{figure}
The Frobenius companion matrix has exactly $n-1$ of the constant entries set to $1$, the remaining $(n-1)^2$ constant entries are zero. 
As noted in~\cite{EKSV}, while a 
companion matrix must have at least $2n-1$ nonzero entries, 
there are companion matrices that have more than $2n-1$ nonzero entries.
As such,
we will say that a companion matrix is \emph{sparse} if it has exactly $2n-1$ nonzero entries. 
 Sparse companion matrices have also been called \emph{intercyclic companion matrices} because of an associated
digraph structure \cite{GSSV}. The second companion matrix in Figure~\ref{frcompanion} is not sparse.
If the $n-1$ nonzero constant entries of a sparse companion are $1$, then we call the companion matrix \emph{unit sparse}. 
 Each unit sparse companion matrix is equivalent to a unit lower triangular
Hessenberg matrix, as will be noted in Theorem~\ref{tsparse}. For $0\leq k\leq n-1$, we say the $k$-th \emph{subdiagonal} of a matrix $A$ 
consists of the entries $\{a_{k+1,1},a_{k+2,2},\ldots,a_{n,n-k}\}$. Note that the $0$-th subdiagonal is usually
called the main diagonal of a matrix. 

\begin{theorem}\cite[Corollary 4.3]{EKSV}\label{tsparse}
$A$ is an $n\times n$ unit sparse companion matrix if and only if $A$ is equivalent to a unit lower 
Hessenberg matrix 
\begin{equation}\label{sparseForm}
C= \left[ \begin{array}{ccccc}
\multirow{2}{*}{$O$} & \multicolumn{2}{c|}{\multirow{2}{*}{$I_{m}$}} & \multicolumn{2}{c}{\multirow{2}{*}{$O$}} \\
                              & \multicolumn{2}{c|}{}                        & \multicolumn{2}{c}{} \\ \hline
\multicolumn{3}{c|}{\multirow{3}{*}{$K$}} & \multicolumn{2}{c}{\multirow{2}{*}{$I_{n-m-1}$}} \\
\multicolumn{3}{c|}{}                        & \multicolumn{2}{c}{} \\
\multicolumn{3}{c|}{}                        & \multicolumn{2}{c}{$O$}\\
\end{array} \right]
\end{equation}
for some $(n-m)\times (m+1)$ matrix $K$ with $m(n-1-m)$ zero entries, 
such that $C$ has $-a_{n-1-k}$ on its $k$th subdiagonal, for $0\leq k \leq n-1$.
\end{theorem}
For example, Figures \ref{fsparse} and \ref{fFiedler} give companion matrices 
with the  structure specified in (\ref{sparseForm}).  
Note that the rectangular matrix $K$ in the Hessenberg form of a sparse companion matrix
has $-a_{n-1}$ in the top right corner and $-a_0$ in the lower left corner.

In 2003, Fiedler~\cite{Fiedler} introduced some companion matrices via products of some block diagonal matrices. 
Each of the matrices introduced by Fiedler is unit sparse and is equivalent to a unit lower Hessenberg
matrix, as was noted in \cite[Corollary 4.4]{EKSV}. We will define the Fiedler matrices using this Hessenberg form.
A \emph{Fiedler companion matrix} is a sparse companion matrix 
which is unit lower Hessenberg with $-a_0$ in position $(n,1)$, 
and if $a_{k-1}$ is in position $(i,j)$ then
$a_k$ is in position $(i-1,j)$ or $(i, j+1)$ for $1\leq k \leq n-1$.
Examples of Fiedler matrices are given in Figure~\ref{fFiedler}.
\begin{figure}[ht]
\[ \left[ \begin{array}{ccc|cc}
0    & 1    & 0    & 0    & 0 \\
0    & 0    & 1    & 0    & 0 \\ \hline
0    & 0    & -a_4 & 1    & 0 \\
0    & -a_2 & -a_3 & 0    & 1 \\
-a_0 & -a_1 & 0    & 0    & 0 \\
\end{array} \right],
\left[ \begin{array}{cccc|c}
0    & 1    & 0    & 0    & 0 \\
0    & 0    & 1    & 0    & 0 \\
0    & 0    & 0    & 1    & 0 \\ \hline
0    & 0    & -a_3 & -a_4 & 1 \\
-a_0 & -a_1 & -a_2 & 0    & 0 \\
\end{array} \right], 
\left[ \begin{array}{cc|ccc}
0    & 1    & 0    & 0    & 0 \\ \hline
0    & -a_4 & 1    & 0    & 0 \\
0    & -a_3 & 0    & 1    & 0 \\
0    & -a_2 & 0    & 0    & 1 \\
-a_0 & -a_1 & 0    & 0    & 0 \\
\end{array} \right] \]
\caption{Fiedler companion matrices}\label{fFiedler}
\end{figure}
Note that the class of Fiedler companion matrices includes the Frobenius companion matrix. 

In~\cite{DDP}, bounds for the roots of a polynomial were determined using the Fiedler companion matrices
based on the products of block diagonal matrices introduced by Fiedler. Here we will take advantage of the Hessenberg form
of the Fiedler companion matrices to provide insight into bounds on polynomial roots and also compare
the bounds developed for Fiedler matrices with new bounds based on the larger class of unit sparse companion matrices.

The bounds will be obtained from matrix norms. A matrix norm is \emph{submultiplicative} if $\|AB\| \leq \|A\| \|B\|$ for all matrices $A,B \in \mathbb{F}^{n \times n}$.  
As noted in \cite{HJ}, if a matrix norm is submultiplicative and $\lambda$ is an eigenvalue of $C$ 
 then 
\begin{equation}\label{bound}
|\lambda| \leq \|C\|.
\end{equation}
Since the eigenvalues of a companion matrix $C(p)$ are the same as the roots of $p$, this technique can be used to find a bound on the roots of $p$.
However, there are many submultiplicative matrix norms and many nonequivalent unit sparse companion matrices to consider. 
As was done in \cite{DDP}, for a matrix $A=[a_{ij}]$ of order $n$, we will focus on two norms that have simple expressions, the $\infty$-norm and the $1$-norm:
$$ \|A\|_{\infty} = \max_{1\leq i \leq n}{\sum_{j=1}^{n}|a_{ij}|} \qquad \textrm{ and} \qquad \|A\|_{1} = \max_{1\leq j \leq n}{\sum_{i=1}^{n}|a_{ij}|}.$$
 We will let $N(A)$ be the minimum of the $\infty$-norm and the $1$-norm,
\begin{equation}\label{upperBound}
 N(A) = \min \{ \|A\|_{\infty},\|A\|_1 \},
\end{equation}
so that $|\lambda| \leq N(A)$
for any matrix $A$ with eigenvalue $\lambda$. Note that if $A$ and $B$ are 
equivalent matrices then $N(A)=N(B)$.

\section{Upper Bounds Based on Sparse Companion Matrices}\label{secStraightNonFied}
Let $p$ be a polynomial as in (\ref{poly}) and $C=C(p)$ 
be an arbitrary unit sparse companion matrix. Recall $\F$ is the Frobenius companion matrix of $p$. 
Let $S_i= \{k\ |\ {-a_k} \text{ is in row }i \text{ of }C \}$ and let $T_i= \{k\ |\ {-a_k} \text{ is in column }i \text{ of }C \}$ for each $i\in \{1,2,\ldots,n\}$. Note that $0 \in S_n$ and $0 \in T_1$. 
Using the norms $\|C\|_\infty$, $\|C\|_1$, $\|\F\|_\infty$, and $\|\F\|_1$ respectively, the following 
known bounds on a root $\lambda$ of $p$ are obtained from (\ref{bound}):
\begin{align}\label{4bounds}
%\|C\|_\infty   
|\lambda| &\leq \max \left\{ \sum\limits_{i \in S_n} |a_i|,\ 1+\sum\limits_{i \in S_{n-1}} |a_i|,\ \ldots,\ 1+\sum\limits_{i \in S_1} |a_i| \right\}, \nonumber \\
%\|C\|_1            
|\lambda|&\leq \max \left\{ \sum\limits_{i \in T_1} |a_i|,\ 1+\sum\limits_{i \in T_2}|a_i|,\ \ldots,\ 1+\sum\limits_{i \in T_n}|a_i| \right\}, \\
%\|\F\|_\infty
|\lambda|&\leq \max \left\{ 1,\ |a_0|+|a_1|+\cdots+|a_{n-1}| \right\}, {\rm \quad and} \nonumber \\
%\|\F\|_1      
|\lambda|&\leq \max \left\{ |a_0|,\ 1+|a_1|,\ 1+|a_2|,\ \ldots,\ 1+|a_{n-1}| \right\}. \nonumber
\end{align}

We first note that, in order for $N(C)$ to be an improvement over $N(\F)$ as a bound on the roots of
$p$, it is necessary that the constant term of $p$ be less than $1$:
\begin{theorem}\label{aNoughtLess1}
Let $\F$ be the Frobenius companion matrix and let $C\neq \F$ be a  unit sparse companion matrix, both based on the polynomial $p$ in \textnormal{(\ref{poly})}.
If $N(C) < N(\F)$ then $|a_0| < 1$.
\end{theorem}
\begin{proof}
Suppose $|a_0| \geq 1$. Let $M=\max\{ |a_k|\ \vert\ 1\leq k\leq n-1\}.$

Case 1: Suppose $|a_0|\geq 1+M$. Then $|a_0| \geq 1+|a_i|$ for all $i\in \{1,2,\ldots,n-1\}$. Then $N(\F)=|a_0|$. Further $\sum\limits_{i\in S_n}|a_i| \geq |a_0|$ and $\sum\limits_{i\in T_1}|a_i| \geq |a_0|$, so $\|C\|_\infty \geq N(\F)$ and $\|C\|_1 \geq N(\F)$. Therefore $N(C) \geq N(\F)$.

Case 2: Suppose $1+M > |a_0|$. Then $\|\F\|_1=1+M$ and $N(\F) = 1+M$. Note that every row and column of $C$ contains either $1$ or $|a_0| \geq 1$. Therefore $N(C) \geq 1+M = N(\F)$.
\end{proof}

The next theorem provides necessary conditions on the shape of the companion matrix $C$ when $N(C)$ provides
an improved bound on $N(\F)$.

\begin{theorem}\label{maxPos}
Let $\F$ be the Frobenius companion matrix and let $C\neq \F$ be a unit sparse companion matrix, 
both based on the polynomial $p$ in \textnormal{(\ref{poly})}.
Let $M=\max\{ |a_k|\ \vert\ 1\leq k\leq n-1\}$.
If $N(C) < N(\F)$, then either \begin{enumerate}[label=\normalfont{\arabic*.}]
\item All coefficients with $|a_i|=M$, $i\in\{ 1,2,\ldots,n-1 \}$, are in the $n$-th row of $C$ and $\|C\|_{\infty} < N(\F) \leq \|C\|_{1}$, or
\item All coefficients with $|a_i|=M$, $i\in\{ 1,2,\ldots,n-1 \}$, are in the $1$st column of $C$ and $\|C\|_{1} < N(\F) \leq \|C\|_{\infty}$.
\end{enumerate}
\end{theorem}
\begin{proof}
Suppose $N(C) < N(\F)$. By Theorem \ref{aNoughtLess1}, $|a_0| < 1$ and so $\|\F\|_1 = 1+M$ and $N(\F) \leq 1+M$.

Suppose there exists a coefficient with $|a_i|=M$, $i\in \{ 1,2,\ldots,n-1 \}$, that is not in the $n$-th row or the $1$st column of $C$. Then $\|C\|_{\infty} \geq 1+M$, $\|C\|_{1} \geq 1+M$, and $N(C) \geq 1+M$. Since $N(\F) \leq 1+M$, $N(C) \geq N(\F)$, which is a contradiction to $N(C)<N(\F)$.

Likewise, there would be a contradiction if there exist two coefficients (not including $a_0$) with modulus $M$ such that
one is in the first column of $C$ and one is in the last row of $C$.

Suppose all coefficients with $|a_i|=M$, $i\in \{ 1,2,\ldots,n-1 \}$, are in the $n$-th row. Then $\|C\|_{1} \geq 1+M.$ Therefore $\|C\|_{1} \geq N(\F).$ Since $N(C) < N(\F) \leq \|C\|_{1}$, $\|C\|_{\infty} < N(\F) \leq \|C\|_1$.

The proof of part 2 is similar with the $1$-norm and $\infty$-norm reversed.
\end{proof}

The next result demonstrates that if $N(C)$ provides a better bound than $N(\F)$ for a polynomial $p$, then either
$p$ does not have many coefficients with maximum modulus, or the maximum modulus 
is small.

\begin{corollary}\label{MIsLimited}
Suppose $M=\max\{ |a_k|\ \vert\ 1\leq k\leq n-1\}$
and $u = \vert\{ k \ \vert \ M=|a_k| \}\vert$.
If $N(C) < N(\F)$ for some unit sparse companion matrix $C$, then $(u-1)M < 1-|a_0|$.
\end{corollary}
\begin{proof}
Suppose $N(C) < N(\F)$. Then by Theorem~\ref{maxPos} all coefficients with $|a_i|=M$, $i\in \{ 1,2,\ldots,n-1 \}$, 
must be in the $n$-th row of $C$ and $\sum_j |C_{ij}|<1+M$ for each row $i$ of $C$.
Working with the $n$-th row,
\begin{align*}
M+M+\cdots +M + |a_0| &< 1+M.
\end{align*}
Hence $(u-1)M < 1-|a_0|$.
\end{proof}

The last results have described various conditions under which a unit sparse companion matrix might produce a better bound than the Frobenius companion matrix. When the unit sparse companion matrix provides a better bound, the Frobenius bound $N(\F)$ is at most twice as large as $N(C)$ and can exceed $N(C)$ by no more than 
one. This mimics what was discovered about Fiedler matrices in \cite[Theorem 4.3]{DDP}.

\begin{theorem}\label{twoIsBestRatio}
Let $\F$ be the Frobenius companion matrix and let $C\neq \F$ be a unit sparse companion matrix, both based on the polynomial $p$ in \textnormal{(\ref{poly})}. Then $N(\F)<2N(C)$
and $N(\F) - N(C) \leq 1$. The latter inequality is strict if $a_0 \neq 0$.
\end{theorem}
\begin{proof}
Suppose $N(C) < N(\F)$. Let $M=\max\{ |a_k|\ \vert\ 1\leq k\leq n-1\}$. Then
\[ N(\F) = \min \left\{ \max \left\{ |a_0|,1+M \right\}, \max \left\{ 1,|a_0|+|a_1|+\cdots+|a_{n-1}| \right\} \right\}. \]
But since $|a_0| < 1$,
\[
N(\F) = \min \left\{ 1+M, \max \left\{ 1,|a_0|+|a_1|+\cdots+|a_{n-1}| \right\} \right\}.
\]
Hence $N(\F) \leq 1+M$. Similarly, with $u = \vert\{ k \ \vert \ M=|a_k| \}\vert,$
\[ N(C) \geq |a_0|+uM. \]
Thus 
\begin{align*}
N(\F) - N(C) &\leq 1+M-|a_0|-uM \\
&\leq 1-|a_0|-(u-1)M\leq 1.
\intertext{Solving for the ratio gives}
\dfrac{N(\F)}{N(C)} &\leq \dfrac{1}{N(C)} + 1 \leq 2 
\end{align*}
since $N(C)\geq 1$. 

Suppose $N(\F)=2N(C)$. 
Given that $N(\F)\geq N(C)+1$, and $N(C)\geq 1$, it follows that $N(C) = 1$. According to Theorem \ref{maxPos}, only one of $\|C\|_\infty$ and $\|C\|_1$ can be sharper than $N(\F)$. Thus by equivalence, we may assume $\|C\|_\infty$. In this case $\|C\|_\infty = 1$.  Then all coefficients not in the $n$-th row of $C$ are $0$, otherwise $\|C\|_\infty \geq 1+|a_i|$ where $a_i$ is any nonzero coefficient not in the $n$-th row of $C$. However, if the nonzero coefficients are all in the $n$-th row of $C$, $C$ is a Frobenius matrix. Then $N(C) = N(\F) = 1$, contradicting $N(\F)=2N(C)$. 
Therefore $N(\F)< 2N(C)$. 
\end{proof}

When seeking an optimal upper bound $N(A)$ over all unit sparse companion matrices $A$ using the Hessenberg structure, then one 
can restrict attention to the $\infty$-norm. In particular, if $A$ is a sparse companion matrix in Hessenberg form 
with $N(A)=\| A\|_1$ 
then $B=RA^TR$ (obtained from $A$ by a reflection across the antidiagonal) is also Hessenberg
and $B$ is equivalent to $A$ with $N(B)=\|B\|_\infty =N(A)$. 
In the next section we restrict our attention to the $\infty$-norm.

\section{Upper Bounds Based on Fiedler Companion Matrices}\label{secFLshape}
In this section, using the Hessenberg structure, for a given polynomial, we note that there is a particular Fiedler matrix 
which provides the best upper bound for $N(C)$ over all Fiedler
matrices $C$. Given $0\leq b\leq n-1$, let 
\[ L_b = \left[
\begin{array}{ccc|c}
\multirow{2}{*}{$O$} & \multicolumn{2}{c|}{\multirow{2}{*}{$I_b$}}            & \multirow{2}{*}{$O$} \\
                              & \multicolumn{2}{c|}{}                              &                    \\ \hline
     &        & -a_{n-1} & \multirow{3}{*}{$I_{n-b-1}$} \\
     & O      & \vdots   & \phantom{\qquad}   \\
     &        & -a_{b+1} &                    \\
-a_0 & \cdots & -a_{b}   & O       \\
\end{array} \right]. \]
Note that $L_b$ is a Fiedler companion matrix and  $L_{n-1} = \F$.
\begin{theorem}\label{LBestFied}
Let $F_b$  be any Fiedler companion matrix (in Hessenberg form) with $a_b$ appearing in 
row $n$ for some $b \in \{0,1,2,\ldots,n-2\}$ but $a_{b+1}$ is in row $n-1$.
Then $\|L_b\|_{\infty} \leq \|F_b\|_{\infty}$.
\end{theorem}
\begin{proof}
Let $K=\max\{ |a_i| \ | \ i>b\}$ and $g=\max\{ i\ |\ |a_i|=K \}.$  
Then \[\|L_b\|_{\infty} = \max \left\{ 1+K,\ \sum\limits_{j=0}^{b}|a_j| \right\}.\]
Note that there are no other coefficients of $p$ in the same row
as $a_g$ in $L_b$ but
there may be more coefficients in the same row as $a_g$ in $F_b$.
Therefore $\|F_b\|_{\infty} \geq 1+K$ and $\|F_b\|_{\infty} \geq \sum\limits_{j=0}^{b}|a_j|$. Thus $\|F_b\|_{\infty}\geq \|L_b\|_{\infty}$.
\end{proof}

Combined with Theorem \ref{maxPos}, Theorem~\ref{LBestFied} implies that when seeking an upper bound with a Fiedler matrix that improves on the Frobenius bounds, one can restrict their attention to the
$\infty$-norm of an $L_b$ Fiedler matrix.
Furthermore, the next theorem indicates the choice of $b$ such that
$L_b$ gives the best upper bound of all Fiedler matrices.

\begin{theorem}\label{LBendBest}
Let $r=\max\left\{k\ \left|  \ \sum\limits_{i=0}^{k-1} |a_i|<1 \right. \right\}$ with $r=0$ if $|a_0|\geq 1$.
Then $\|L_r\|_{\infty} \leq \|L_b\|_{\infty}$ for all $b\in \{0,1,2,\ldots, n-2\}.$
\end{theorem}
\begin{proof}
Suppose $b>r$. Then
\begin{eqnarray*}
\|L_b\|_\infty &=& \max \left\{\sum\limits_{i=0}^b |a_i|,\ 1+|a_{b+1}|, 1+|a_{b+2}|,\ldots, 1+|a_{n-1}| \right\}\\
               &=& \max \left\{\sum\limits_{i=0}^r |a_i|+\sum\limits_{i=r+1}^b |a_i|,\  1+|a_{b+1}|, 1+|a_{b+2}|,\ldots, 1+|a_{n-1}| \right\}\\
               &\geq &  \max\left\{\sum\limits_{i=0}^r |a_i|,\  1+|a_{r+1}|, 1+|a_{r+2}|,\ldots, 1+|a_{n-1}| \right\}%\\ 
               %&=& 
               =\|L_r\|_\infty.
\end{eqnarray*}               
Suppose $b<r$. Then 
\begin{eqnarray*}
\|L_b\|_\infty &=& \max \left\{\sum\limits_{i=0}^b |a_i|,\ 1+|a_{b+1}|, 1+|a_{b+2}|, \ldots, 1+|a_{n-1}| \right\}\\
               &=& \max \bigg\{1+|a_{b+1}|, 1+|a_{b+2}|, \ldots, 1+|a_{n-1}| \bigg\}\\
               &\geq &  \max \bigg\{1+|a_{r}|, 1+|a_{r+1}|,\ldots, 1+|a_{n-1}| \bigg\}\\
               &\geq & \max \left\{ \sum\limits_{i=0}^{r}|a_i|,\ 1+|a_{r+1}|, 1+|a_{r+2}|, \ldots, 1+|a_{n-1}| \right\}%\\
               =  \|L_r\|_\infty.
\end{eqnarray*}
\end{proof}

\begin{example}
Consider the polynomial \[p = x^8 - 0.1x^7 - 0.1x^6 - 0.3x^5 - 0.1x^4 - 0.5x^3 - 0.1x^2 - 0.1x - 0.1.\] To determine the Fiedler matrix that provides the best
bound on the roots of $p$, we simply find the first partial sum $\sum\limits_{i=0}^b |a_i|$ which is more than
one. In particular,
\[ \sum\limits_{i=0}^{4}|a_i| = |-0.1| + |-0.1| + |-0.1| + |-0.5| + |-0.1| < 1, \]
but 
$$ \sum\limits_{i=0}^5|a_i|=1.2 > 1.$$
Thus, by Theorems~$\ref{LBestFied}$ and $\ref{LBendBest}$, $L_5$ gives the best bound of all Fiedler matrices.
For comparison, we calculated $\|L_b\|_\infty$ for each $b\in \{0,1,\ldots,7\}$:
\[
\begin{array}{c||c|c|c|c|c|c|c|c}
b & 0&1&2&3&4&5&6&7\\ \hline 
\|L_b\|_\infty&1.5&1.5 &1.5&1.3&  1.3  & 1.2&1.3 & 1.4\\
\end{array}.\]
\end{example}

\medskip

\section{Upper Bounds obtained from extended polynomials}\label{secExtendedPoly}
Sections \ref{secStraightNonFied} and \ref{secFLshape} illustrate some necessary conditions for a unit sparse
companion matrix of $p$  to provide a sharper bound on the roots of $p$ than the Frobenius matrix. By multiplying the polynomial $p$ by the factor $x^q$, for some $q>0$, some of these restrictions can be removed. In particular, if $q>0$, then a root of $p$ will also
be a root of $x^qp$. Thus if
$\lambda$ is a root of $p$, then
$ |\lambda| \leq N(C(x^q p)). $
Using a unit sparse companion matrix of the polynomial $x^qp$ for some $q>0$, instead of the polynomial $p$,  
can provide sharper bounds on a root of $p$ 
than the Frobenius bounds and the companion matrix bounds developed in Section \ref{secStraightNonFied}:

\begin{example}\label{extendedIsBest}
Consider the polynomial $p = x^6 - x^5 - 2x^4 - x^3 - 4x^2 - 2x - 3$. 
By Theorem~$\ref{aNoughtLess1}$, $N(C)\geq N(\F(p)) =5$ if $C=C(p)$ is a companion matrix of $p$. 
Consider a companion matrix of the polynomial $x^3 p$:
\begin{equation}\label{x3}
 C(x^3 p) = \left[ \begin{array}{cccccc|ccc}
0 & 1 & 0 & 0 & 0 & 0 & 0 & 0 & 0 \\
0 & 0 & 1 & 0 & 0 & 0 & 0 & 0 & 0 \\
0 & 0 & 0 & 1 & 0 & 0 & 0 & 0 & 0 \\ \cdashline{4-9}
0 & 0 & 0 & \lline{0} & 1 & 0 & 0 & 0 & 0 \\
0 & 0 & 0 & \lline{0} & 0 & 1 & 0 & 0 & 0 \\ \hline 
0 & 2 & 0 & \lline{0} & 0 & 1 & 1 & 0 & 0 \\
0 & 0 & 0 & \lline{0} & 1 & 2 & 0 & 1 & 0 \\
0 & 0 & 3 & \lline{0} & 0 & 0 & 0 & 0 & 1 \\
0 & 0 & 0 & \lline{0} & 0 & 4 & 0 & 0 & 0
\end{array} \right]. 
\end{equation}
The solid lines in $(\ref{x3})$ indicate the partition of the companion matrix $C(x^3 p)$ as outlined in $(\ref{sparseForm})$ and the dashed lines illustrate the fact that $C(x^3 p)$ has extra columns (and rows) compared to $C(p)$. Note that $\|C(x^3 p)\|_\infty = 4$ and so $\|C(x^3 p)\|_\infty < N(\F(p))$. Thus, using the polynomial $x^3 p$ we obtain a tighter upper bound on a root of $p$ than what can be obtained by using a companion matrix of $p$ itself. 
\end{example}

The companion matrices of $x^q p$ can be used to generalize the bound found in \cite[Theorem 3.3]{DDP}.
Note that given any partition $P = \{ P_1,P_2,\ldots,P_t \}$ of $\{ 0,1,\ldots,n-1 \}$, 
by the pigeonhole principle, there exists a $P_\ell$ such that $\max\{i\ \vert\ i \in P_\ell\} \leq n-t$. 
By relabeling, we can assume $\ell=1$.

\begin{theorem}\label{extendedBound}
Let $\lambda$ be a root of $p$ as in $(\ref{poly})$. Let $P = \{ P_1,P_2,\ldots,P_t \}$ be a partition of $\{ 0,1,\ldots,n-1 \}$
into $t$ nonempty sets, 
with $\max\{i\ \vert\ i \in P_1\} \leq n-t$. 
Then
\[ |\lambda| \leq B_P = \max\left\{ \sum\limits_{i\in P_1} |a_i|,1+\sum\limits_{i\in P_2} |a_i|,\ldots,1+\sum\limits_{i\in P_t} |a_i| \right\}.
\]
\end{theorem}
\begin{proof}  
Relabel the  partition $P = \{ P_1,P_2,\ldots,P_t \}$ of $\{ 0,1,\ldots,n-1 \}$
with $\max\{i\ \vert\ i \in P_1\} \leq n-t$ so that
$\max\{i\ \vert\ i\in P_v \} > \max\{i\ \vert\ i\in P_u \}$ for all $v > u > 1$. Then $n-1\in P_t$, $n-2\in P_t\cup P_{t-1}$, and inductively, 
$n-r \in P_t\cup P_{t-1} \cup \cdots \cup P_{t-r+1}$ for $1\leq r\leq n-1$. In particular,
if $i\in P_j$ then $i\leq n+j-t-1$. 

We can construct an order $n+t-1$ companion matrix $C=C(x^{t-1}p)$ in the form (\ref{sparseForm}) with
$K$ a $t\times n$ submatrix. In particular, the order of $C$ is sufficiently large so that if $i\in P_j$ then $a_i$
can be in row $n+t-j$ of $C$: place $a_i$ in position $(n+t-j, t+i-j+1)$ so
that $a_i$ is on the $(n-i-1)$th subdiagonal of $C$. 
Note that $n\leq n+t-j\leq n+t-1$ since $1\leq j\leq t$, and $1\leq t+i-j+1\leq n$ since $1\leq j \leq t$ and $i\leq n+j-t-1$,
so that $a_i$ is in submatrix $K$. Now $B_p=\|C\|_\infty$.
\end{proof}

With the partition relabeling as noted in the above proof, we can observe that $a_{n-1}\in P_t$. Consequently
the proof used $C=C(x^qp)$ with $q=t-1$ in order to ensure there are sufficient columns so that $a_0$ could
appear in the same row as $a_{n-1}$ while satisfying the diagonal conditions of
the matrix $K$ in $(\ref{sparseForm})$.  

Theorem~\ref{extendedBound} is a generalization of the four bounds in $(\ref{4bounds})$, including the Frobenius bound $N(\F)$, as well as the bounds in \cite[Theorem 3.3]{DDP}. Particularly, the two portions 
of the Frobenius bound are obtained from $\|C\|_\infty$ with
 $C=C(x^{t-1}p)$, using $t=n$ and $t=1$ respectively.

\begin{example}\label{bigPartitionEx}
Let $p = x^8 - x^7 - 3x^6 - 2x^4 - 2x^4 - 4x^3 - 3x^2 - 5x - 3$, and let $\lambda$ be a root of $p$. Then $N(\F(p)) = 6$, and by Theorem~$\ref{aNoughtLess1}$, $N(\F(p)) \leq N(C(p))$ for all unit sparse companion matrices of $p$.

Consider the partition: $P = \{ \{ 0,7 \}, \{ 1 \}, \{ 2 \}, \{ 3 \}, \{ 4 \}, \{ 5 \}, \{ 6 \} \}$. Using the notation in the proof of Theorem~$\ref{extendedBound}$, $t = 7$ and $n-t = 1$. Since $\max\{ \{ 1 \} \} \leq 1$, $P_1 = \{ 1 \}$. Then
\begin{align*}
B_P &= \max\left\{ |a_1|,1+|a_2|,1+|a_3|,1+|a_4|,1+|a_5|,1+|a_6|,1+|a_0|+|a_7| \right\} \\
&= \max\left\{ 5,4,5,3,3,4,5 \right\} = 5.
\end{align*}
Equivalently,  $B_P =\|C(x^6p)\|_\infty$ for
\[ C(x^6p) = \left[
\begin{array}{cccccccc|cccccc}
 0 & 1 & 0 & 0 & 0 & 0 & 0 & 0 & 0 & 0 & 0 & 0 & 0 & 0 \\
 0 & 0 & 1 & 0 & 0 & 0 & 0 & 0 & 0 & 0 & 0 & 0 & 0 & 0 \\
 0 & 0 & 0 & 1 & 0 & 0 & 0 & 0 & 0 & 0 & 0 & 0 & 0 & 0 \\
 0 & 0 & 0 & 0 & 1 & 0 & 0 & 0 & 0 & 0 & 0 & 0 & 0 & 0 \\
 0 & 0 & 0 & 0 & 0 & 1 & 0 & 0 & 0 & 0 & 0 & 0 & 0 & 0 \\
 0 & 0 & 0 & 0 & 0 & 0 & 1 & 0 & 0 & 0 & 0 & 0 & 0 & 0 \\ \cdashline{7-14}
 0 & 0 & 0 & 0 & 0 & 0 &\lline{0} & 1 & 0 & 0 & 0 & 0 & 0 & 0 \\  \hline
 3 & 0 & 0 & 0 & 0 & 0 &\lline{0} & 1 & 1 & 0 & 0 & 0 & 0 & 0 \\
 0 & 0 & 0 & 0 & 0 & 0 &\lline{0} & 3 & 0 & 1 & 0 & 0 & 0 & 0 \\
 0 & 0 & 0 & 0 & 0 & 0 &\lline{0} & 2 & 0 & 0 & 1 & 0 & 0 & 0 \\
 0 & 0 & 0 & 0 & 0 & 0 &\lline{0} & 2 & 0 & 0 & 0 & 1 & 0 & 0 \\
 0 & 0 & 0 & 0 & 0 & 0 &\lline{0} & 4 & 0 & 0 & 0 & 0 & 1 & 0 \\
 0 & 0 & 0 & 0 & 0 & 0 &\lline{0} & 3 & 0 & 0 & 0 & 0 & 0 & 1 \\
 0 & 0 & 0 & 0 & 0 & 0 &\lline{0} & 5 & 0 & 0 & 0 & 0 & 0 & 0
\end{array} \right]. \]
\end{example}

In Example~\ref{bigPartitionEx}, $t$ is large relative to $n$. In general, it is good to take 
$t$ as large as possible in order to reduce the size of the parts in the partition, since the sums in $B_P$ are minimized if the parts are small. However, if we intend for the bound $B_P$ to be sharper than a Frobenius bound, 
then the coefficients with maximum modulus need to be in $P_1$, which limits $t$. In particular, 
if $M=\max\{ |a_k|\ \vert\ 1\leq k\leq n-1\}$ and $|a_i|=M$, then we would take $t\leq n-i$.
\begin{example}\label{littlePartitionEx}
Let $p = x^8 - x^7 - 3x^6 - x^5 - 6x^4 - x^3 - 5x^2 - 3x - 1$, and let $\lambda$ be a root of $p$. Then $|\lambda|\leq N(\F(p)) = 7$, and by Theorem~$\ref{aNoughtLess1}$, $N(\F(p)) \leq N(C(p))$ for all unit sparse companion matrices $C(p)$.

Consider the partition: $P = \{ \{ 4 \}, \{ 2 \}, \{ 0,5,6 \}, \{ 1,3,7 \} \}$. With this partition, $t = |P| = 4$ and $n-t = 4$. 
Thus, by Theorem~$\ref{extendedBound}$, we could choose $P_1$ to be $\{2\}$ or $\{4\}$. With
$P_1 = \{ 4 \}$, we have
\begin{align*}
B_P &= \max\left\{ |a_4|,1+|a_2|,1+|a_0|+|a_5|+|a_6|,1+|a_1|+|a_3|+|a_7| \right\} % \\
%&= 
=\max\left\{ 6,6,6,6 \right\} = 6<7,
\end{align*}
and $|\lambda| \leq 6$.
\end{example}

The restrictions on $|a_0|$ provided by Theorem \ref{aNoughtLess1} and Corollary \ref{MIsLimited}
do not apply to $C(x^{q}p)$, $q\geq 1$, as they do to $C(p)$. For instance, Example~\ref{littlePartitionEx} 
illustrates that there is a unit sparse companion matrix $C$ with $N(C(x^4p))<N(\F(p))=N(\F(x^4p))$ 
but $|a_0|=1$. 
However, $C(p)$ can be replaced by $C(x^{q}p)$ in 
Theorems \ref{maxPos} and \ref{twoIsBestRatio} without changing the conclusion. 
In particular, $N(\F(p))<2N(C(x^{q}p))$ and $N(\F(p))-N(C(x^{q}p))\leq 1$.

\section{Lower Bounds Using Monic Reversal Polynomials}\label{secReversal}
One tool for obtaining lower bounds on the roots of a polynomial is to apply the techniques previously developed to the monic reversal polynomial,
as shown in \cite{HJ}. Assuming that $a_0 \neq 0$, the \emph{monic reversal polynomial} of $p$ is 
\begin{equation}\label{reversalPoly}
p^{\sharp}(x) = \dfrac{x^n}{a_0} p\left(x^{-1}\right) = x^n + \dfrac{a_1}{a_0} x^{n-1} + \dfrac{a_2}{a_0} x^{n-2} + \cdots + \dfrac{a_{n-1}}{a_0} x + \dfrac{1}{a_0}.
\end{equation}
In the case that $a_0 = 0$, one could consider the monic reversal of the lower degree polynomial $p/x$ if $a_{1}\neq 0$. 
Note that 
$ (p^\sharp)^\sharp = p.$
Since the roots of $p^{\sharp}$ are the reciprocals of the roots of $p$, the eigenvalues of a companion matrix $C(p^{\sharp})$ are the reciprocals of the roots of $p$. Therefore, if $\lambda$ is a root of $p$,
\begin{equation}\label{lower}
\dfrac{1}{\|C(p^\sharp)\|} \leq |\lambda| 
\end{equation}
for any submultiplicative matrix norm. Including (\ref{bound}), 
we have 
\begin{equation}\label{lowerBoundReversal}
\dfrac{1}{N(C(p^\sharp))} \leq |\lambda| \leq N(C(p)).
\end{equation}
Thus, the theorems from Sections \ref{secStraightNonFied} and \ref{secFLshape} 
can be applied to companion matrices of monic reversal polynomials to obtain lower bounds. For example, by Theorem~\ref{aNoughtLess1}, if $N(C(p^\sharp)) < N(\F(p^\sharp))$ then $|a_0| > 1$.

\begin{example}
Let $\lambda$ be a root of \[ p = x^6 -3x^5 -10x^4 - 4x^3 -5x^2 -x -5. \] 
Note that \[ p^\sharp = x^6 +\frac{1}{5}x^5 +x^4 +\frac{4}{5}x^3 +2x^2 +\frac{3}{5}x -\frac{1}{5}. \] The Frobenius matrix provides a lower bound, $|\lambda| \geq N(\F(p^\sharp))^{-1} = 1/3$, but the companion matrix
\[ C(p^\sharp) = \left[ \begin{array}{ccc|ccc}
0   & 1   & 0   & 0   & 0   & 0   \\
0   & 0   & 1   & 0   & 0   & 0   \\ \hline
0   & 1   & -\nicefrac{1}{5} & 1   & 0   & 0   \\
0   & -\nicefrac{4}{5} & 0   & 0   & 1   & 0   \\
-\nicefrac{3}{5} & 0   & 0   & 0   & 0   & 1   \\
\ \nicefrac{1}{5} & 0   & -2   & 0   & 0   & 0
\end{array} \right], \]
provides a lower bound of $|\lambda| \geq N(C(p^\sharp))^{-1} = 5/11.$
\end{example}

The next theorem implies that of all the unit sparse companion matrices of $p$, the Frobenius matrix will provide either the sharpest lower bound, the sharpest upper bound, or both.

\begin{theorem}\label{chooseUpperOrLower}
Let $C(p)$ and $C(p^\sharp)$ be unit sparse companion matrices for polynomials of the form \textnormal{(\ref{poly})} and \textnormal{(\ref{reversalPoly})} respectively. Let $\F(p)$ and $\F(p^\sharp)$ be the corresponding Frobenius companion matrices. 
\begin{enumerate}[label=(\roman*)]
\item \label{choose&UpperLoses} If $|a_0| \geq 1$ then $N(\F(p)) \leq N(C(p))$.
\item \label{choose&LowerLoses} If $|a_0| \leq 1$ then $N(\F(p^\sharp)) \leq N(C(p^\sharp))$.
\end{enumerate}
\end{theorem}
\begin{proof}
Part \ref{choose&UpperLoses} is the contrapositive of Theorem \ref{aNoughtLess1}. Part \ref{choose&LowerLoses} is an application of Theorem \ref{aNoughtLess1} to the lower bound in (\ref{lowerBoundReversal}).
\end{proof}

\begin{example}
Let $\lambda$ be a root of $$p = x^7 -4x^6 -x^5 -2x^4 -6x^3 -5x^2 -2x -0.2.$$ The Frobenius matrix will give the upper bound  $|\lambda|\leq N(\F)=7$. Using the companion matrix
\[ \small{
C = \left[ \begin{array}{rccc|ccc}
0   & 1 & 0 & 0 & 0 & 0 & 0 \\
0   & 0 & 1 & 0 & 0 & 0 & 0 \\
0   & 0 & 0 & 1 & 0 & 0 & 0 \\ \hline
0   & 0 & 1 & 4 & 1 & 0 & 0 \\
5   & 0 & 0 & 0 & 0 & 1 & 0 \\
2   & 0 & 0 & 2 & 0 & 0 & 1 \\
0.2 & 0 & 0 & 6 & 0 & 0 & 0 \\
\end{array} \right], }
\]
we can improve the Frobenius bound to $|\lambda|\leq N(C)=6.2$.
Since $|a_0| = 0.2 < 1$, Theorem $\ref{chooseUpperOrLower}  \ref{choose&LowerLoses}$
implies that the best lower bound will be obtained from the Frobenius matrix
of the monic reversal  polynomial
\[p^\sharp = x^7 +10x^6 +25x^5 +30x^4 + 10x^3 +5x^2 +20x -5.\]
In particular, using unit sparse companion matrices of $p$, the best lower bound 
we can obtain from $(\ref{lowerBoundReversal})$ is $$|\lambda|\geq N(\F(p^\sharp))=\frac{1}{31}.$$
\end{example}

\begin{corollary}\label{aNoughtFrobDoubleSharp}
Let $\lambda$ be a root of $p$.
If $|a_0| = 1$, of all the unit sparse companion matrices of $p$,
the Frobenius companion matrix provides both 
the sharpest lower bound and the sharpest upper bound on $\lambda$ in $(\ref{lowerBoundReversal})$.
\end{corollary}
Note that the converse of Corollary~\ref{aNoughtFrobDoubleSharp} is not true:

\begin{example}
Consider the polynomial \[p_6 = x^6 -x^5 -x^4 -x^3 -2x^2 -2x -0.2.\] 
By Corollary $\ref{MIsLimited}$, $N(\F) \leq N(C)$ for any unit sparse companion matrix of $p$ since $(u-1)M = 2 \geq 0.8 = 1-|a_0|$. Since $|a_0| = 0.2 < 1$, by Theorem \ref{chooseUpperOrLower} \ref{choose&LowerLoses},  the Frobenius matrix also provides the sharpest lower bound of all the unit sparse companion matrices of $p$.
\end{example}

As in the previous section, we note that Theorem~\ref{twoIsBestRatio} still applies: in particular, while 
$ N(C(p^\sharp))$ can be a better bound than $N(\F(p^\sharp))$, it will not be an improvement of more
than a factor of two.
 
\section{Bounds Using Inverse Matrices and Monic Reversal Polynomials}\label{secInverseReversals}
If $\lambda$ is a root of $p$ and $C$ is a unit sparse companion matrix of $p$, it follows that
\[ \dfrac{1}{|\lambda|} \leq \|C^{-1}\| \]
for any submultiplicative matrix norm. Using this observation with the monic reversal polynomial, we obtain
the upper and lower bounds:
\[ 
\dfrac{1}{N(C^{-1}(p))} \leq |\lambda| \leq N(C^{-1}(p^\sharp)).  
\]

Given $0\leq c\leq n-1$ and $a_0\neq 0$, 
then a unit sparse companion matrix can be partitioned as 
\begin{equation}\label{eC}
C=C(p) = \left[ \renewcommand{\arraystretch}{1.3}
\begin{array}{c|c|cc}
\multirow{2}{*}{$O$} & \multirow{2}{*}{$I_c$} & \multirow{2}{*}{$O$} \\
& & \\ \hline
\multirow{2}{*}{$\textbf{u}$} & \multirow{2}{*}{$H$} & \multirow{2}{*}{$I_{n-c-1}$} \\
& & \\ \hline
-a_0 & \textbf{y}^T & O \\
\end{array} \right],
\end{equation}
for some vectors $\mathbf{u}$ and $\mathbf{y}$, and some $(n-c-1)\times c$ matrix $H$, whose nonzero entries are coefficients of $-p$. 
In this case, the inverse of $C$ will be 
\begin{equation}\label{eC2}
C^{-1} = \left[ \renewcommand{\arraystretch}{1.8}
\begin{array}{c|c|c}
{\frac{1}{a_0}\textbf{y}^T} & O & -\frac{1}{a_0}  \\ \hline
\multirow{2}{*}{$I_c$} & \multirow{2}{*}{$O$} & \multirow{2}{*}{$O$} \\
 & & \\ \hline
\multirow{2}{*}{$\frac{1}{-a_0}\textbf{u}\textbf{y}^T-H$} & \multirow{2}{*}{$I_{n-c-1}$} & \multirow{2}{*}{$\frac{1}{a_0}\textbf{u}$} \\
 & & \\
\end{array} \right]. 
\end{equation}
As noted in~\cite{DDP}, for the Frobenius matrix $\F$,  $\F(p^\sharp)^{-1}$ is equivalent to $\F(p)$. Hence we will continue to compare new bounds to $N(\F)$.

Due to the number of nonequivalent unit sparse companion matrices, 
determining a simple expression for the $1$-norm and the $\infty$-norm of 
the inverse of a companion matrix is not straightforward. In order to simplify matters, 
we will restrict our attention to those with 
$\textbf{u} = \textbf{0}$. In particular, for $1\leq c\leq n-2$, we say a matrix is of \emph{type} $E_c=E_c(p)$ if it
has the form (\ref{eC}) with $\mathbf{u}=\mathbf{0}$. In this case, $\mathbf{y}_1=-a_1$.
Further, we say a matrix $A$ is of \emph{type} $E_c(p^\sharp)^{-1}$ if $A^{-1}$ is of type $E_c(p^\sharp)$. 
In this case, $\mathbf{y}_1 =-a_{n-1}$.

There is only one matrix of type $E_1(p^\sharp)^{-1}$, namely  
\begin{equation}\label{superFiedler}
W =
\left[ \renewcommand{\arraystretch}{1.3}
\begin{array}{c|c|c}
-a_{n-1}        & O			                 & -a_0                          \\ \hline
1               & O 			             &   0                          \\ \hline
{a_1}/{a_0}     & \multirow{4}{*}{$I_{n-2}$} & \multirow{4}{*}{$O$} \\
{a_2}/{a_0}     &                            &                               \\
\vdots          &                            &                               \\
{a_{n-2}}/{a_0} &                            &                               \\
\end{array} \right].
\end{equation}
(The matrix $W$ is equivalent to a matrix labeled $F(p^\sharp)^{-1}$ in \cite{DDP}.) 
When restricting to matrices $A$ of type $E_c(p^\sharp)^{-1}$ obtained from Fiedler matrices, 
\cite[Theorem~5.4]{DDP} compares $\|W\|_\infty$ and $\|\F\|_\infty$
to $\|A\|_\infty$ under various conditions. 
Theorem~\ref{bestEasilyInverted} expands the comparison to all matrices of type $E_{c}(p^\sharp)^{-1}$, including
those obtained from the larger class of unit sparse companion matrices, and once again  
highlights $W$ when $|a_0|>1$. 
Later, in Theorem~\ref{superFiedlerIff}, 
we characterize when bounds derived from $W$ improve those derived from $\F$. 

\begin{theorem}\label{bestEasilyInverted}
Let $M=\max\{ |a_k|\ \vert\ 1\leq k\leq n-1\}$.
Suppose $A$ is of type $E_c(p^\sharp)^{-1}$ for some $c$, $1\leq c\leq n-2$.
\begin{enumerate}[label=\normalfont{\alph*)}]
\item\label{aNought=0} If $|a_0|=1$ then $N(\F) \leq N(A)$. 

\item\label{superFiedlerIsBest} If $|a_0|>1$ then $\|W\|_\infty \leq \|A\|_\infty$. 

\item\label{FrobeniusIsBest} If $|a_0|>1$ and $|a_{n-1}|=M$ then $N(\F)\leq N(A)$.

\item\label{aNought<0} If $|a_0|<1$ then 
\begin{enumerate}[label=\normalfont{\roman*)}]
\item\label{diff1} $N(\F)-\|A\|_{\infty} \leq 1$,
\item\label{ratio2}  $N(\F)\leq 2 \|A\|_{\infty}$, and 
%$\dfrac{N(\F)}{\|A\|_{\infty}} \leq 2$, and
\item $N(\F) \leq \|A\|_{1}$.
\end{enumerate}
\end{enumerate}

\end{theorem}
\begin{proof}
Suppose $1\leq c\leq n-2$ and  $A$ is of type $E_c(p^\sharp)^{-1}$. Then
\[ A = \left[ \begin{array}{cc|cc|c}
\multicolumn{2}{c|}{a_0\textbf{y}^T} & \multicolumn{2}{|c|}{O} & -a_0 \\ \hline
\multicolumn{2}{c|}{\multirow{2}{*}{$I_c$}} & \multicolumn{2}{|c|}{\multirow{2}{*}{$O$}} & \multirow{2}{*}{$O$} \\
& & & & \\ \hline
\multicolumn{2}{c|}{\multirow{2}{*}{$-H$}} & \multicolumn{2}{|c|}{\multirow{2}{*}{$I_{n-c-1}$}} & \multirow{2}{*}{$O$} \\
& & & & \\
\end{array} \right]. \]
for some $\mathbf{y}$, and some $(n-c-1)\times c$ matrix $H$, whose nonzero entries are  
in $\left\{\frac{a_1}{a_0},\frac{a_2}{a_0},\ldots,\frac{a_{n-2}}{a_0}\right\}$, with $a_0\mathbf{y}_1=-a_{n-1}$.
\begin{enumerate}[label=\alph*)]

\item Suppose $|a_0|=1$. Then there is a $1$ in every row and column of $A$. Also, the nonzero values in $H$ are coefficients of $p$. 
Therefore $N(A) \geq 1+M$. Since $|a_0|=1$,
\begin{align*}
N(\F) &= \min\left\{ \max\left\{ |a_0|,1+M \right\}, \max\left\{ 1,|a_0|+|a_1|+\cdots+|a_{n-1}| \right\} \right\} \\
&= 1+M.
\end{align*}
Therefore $N(\F) = 1+M \leq N(A)$.

\item Suppose $|a_0|>1$. Note that $\|W\|_\infty = \max\left\{ |a_0|+|a_{n-1}|,1+ \dfrac{M}{|a_0|} \right\}$. If $|a_0|+|a_{n-1}| \geq 1+ \dfrac{M}{|a_0|}$ then  
$\|W\|_\infty = |a_0|+|a_{n-1}| 
\leq \|A\|_\infty$. 
Suppose $\|W\|_\infty= 1+ \dfrac{M}{|a_0|}$. 
If there is a coefficient of $p$ with modulus $M$ in the first row of $A$ then $\|A\|_\infty \geq 1+M> \|W\|_\infty$.
If there is a coefficient of $p$ with modulus $M$ in $H$, then $\|A\|_\infty \geq 1+ \frac{M}{|a_0|} = \|W\|_\infty$.
Therefore  $\|W\|_\infty \leq \|A\|_\infty$.

\item Suppose $|a_0|>1$ and $|a_{n-1}|=M$. Note that $\|A\|_1 \geq |a_0|$ and $\|A\|_\infty \geq |a_0|+|a_{n-1}|$. Since $|a_{n-1}|=M$,  $\|A\|_1 \geq 1+M$ and $\|A\|_\infty \geq |a_0|+M$. Since $|a_0|>1$, $N(\F) = \max \left\{ |a_0|,1+M \right\}$. Therefore $N(\F) \leq N(A)$.

\item Suppose $|a_0|<1$.
\begin{enumerate}[label=\normalfont{\roman*)}]
\item  
Suppose $\|A\|_\infty < N(\F)$. Since $|a_0|<1$,
\begin{align*}
N(\F) &= \min\left\{ \max\left\{ |a_0|,1+M \right\}, \max\left\{ 1,|a_0|+|a_1|+\cdots+|a_{n-1}| \right\} \right\}\\
&= \min\left\{ 1+M,\max\left\{ 1,|a_0|+|a_1|+\cdots+|a_{n-1}| \right\} \right\}. \\
\intertext{However, if $1>|a_0|+|a_1|+\cdots+|a_{n-1}|$ then $N(\F) = 1 \leq \|A\|_\infty$, so}
N(\F) &= \min\left\{ 1+M,|a_0|+|a_1|+\cdots+|a_{n-1}| \right\}.
\end{align*}
In order for $\|A\|_\infty < N(\F)$, all coefficients of $p$ with modulus $M$ must be in the first row of $A$, since $1+M < 1+\dfrac{M}{|a_0|}$ when $|a_0|<1$. This gives
$\|A\|_\infty \geq  |a_0|+uM$, with  $u = \vert\{ k \ \vert \ M=|a_k| \}\vert$.
Therefore, since $N(\F) \leq 1 + M$, 
$N(\F) - \|A\|_\infty \leq 1+M-|a_0|-uM \leq 1$. 

\item Using part (i), we have
$$\dfrac{N(\F)}{\|A\|_\infty} \leq 1+ \dfrac{1}{\|A\|_\infty}. $$
However, $\|A\|_\infty \geq 1$, so  $N(\F)\leq 2\|A\|_\infty$.

\item  
Since $|a_0|<1$, $\|A\|_1 \geq 1+M$. Recall that \[N(\F) = \min\left\{ 1+M, |a_0| + |a_1| + \cdots + |a_{n-1}| \right\} \leq 1+M.\] Therefore
$ N(\F) \leq \|A\|_1. $
\end{enumerate}
\end{enumerate}\end{proof}

Observe that the conditions given in Theorem~\ref{bestEasilyInverted}\ref{aNought<0}\ref{diff1} and \ref{ratio2} mimic conditions given in Theorem~\ref{twoIsBestRatio}. These conditions can also be applied to the polynomials $x^qp$ in Section~\ref{secExtendedPoly} and the lower bounds in Section~\ref{secReversal}. However, 
when $|a_0|>1$, it is possible that $N(\F)>2\|W\|_\infty$ and $N(\F)-\|W\|_\infty>1$. In fact,
as we note in Remark~\ref{large}, the ratio can be made arbitrarily large.

\begin{example}
Let $\lambda$ be a root of $p = x^6 - x^5 + 12x^4 + 6x^3 + 36x^2 + 18x - 6$. $\|W\|_\infty = 7$ and $N(\F) = 37$, giving
\[ N(\F) - \|W\|_\infty = 30 \textnormal{ and } \frac{N(\F)}{\|W\|_\infty} = \frac{37}{7} \doteq 5.3. 
\]
\end{example}

In the next theorem, we characterize when 
$W$ provides a better upper bound on the roots of a polynomial than the Frobenius matrix $\F$.

\begin{theorem}\label{superFiedlerIff}
Given a  polynomial $p$ 
with $M=\max\{ |a_k|\ \vert\ 1\leq k \leq n-1\}$. Then
$\|W\|_\infty < N(\F)$ if and only if $1<|a_0|<1+M-|a_{n-1}|.$
\end{theorem}
\begin{proof}
Suppose $\|W\|_\infty < N(\F)$. Then
\[ \max\left\{ |a_0|+|a_{n-1}|,1+\dfrac{M}{|a_0|} \right\} < \min\left\{ \max\left\{ |a_0|,1+M \right\}, \max\left\{ 1,|a_0|+|a_1|+\ldots +|a_{n-1}| \right\} \right\}. \]
Suppose $|a_0|\leq 1$. Then $\|W\|_\infty\geq 1+M \geq N(\F)$. Thus $|a_0|>1$. 

Suppose $|a_0| \geq 1+M$. Then $N(\F) = |a_0|$. But $\|W\|_\infty \geq |a_0|+|a_{n-1}|$. 
Therefore $|a_0| < 1+M$ and $N(\F) = 1+M$.
Hence 
$|a_0|+|a_{n-1}| < 1+M.$ 
Therefore
$ 1 < |a_0| < 1 + M-|a_{n-1}|.$

Suppose $1<|a_0|<1+M-|a_{n-1}|$. Since $1 < |a_0|$, $1<|a_0|+|a_1|+\cdots+|a_{n-1}|$ and, in fact, $1+M<|a_0|+|a_1|+\cdots+|a_{n-1}|$. 
Thus $N(\F) = 1+M$, since $|a_0|<1+M$. Therefore
$\|W\|_\infty <  N(\F). $
\end{proof}

\begin{theorem}
Given a  polynomial $p$ with $|a_0|>1$. If $N(\F)<\|W\|_\infty$ then $\|W\|_\infty <2 N(\F)$.
\end{theorem}

\begin{proof}
Suppose $|a_0|>1$ and $N(\F)<\|W\|_\infty$.
Let $M=\max\{ |a_k|\ \vert\ 1\leq k\leq n-1\}$. 
By Theorem~\ref{superFiedlerIff}, $|a_0|\geq 1+M-|a_{n-1}|$. 
%	Suppose $|a_0|>1+M-|a_{n-1}|$. 
Then %$|a_0|>1$ and
	$$|a_0|+|a_{n-1}|\geq 1+M>1+\dfrac{M}{|a_0|}.$$
	Thus 
	$$\|W\|_\infty = \max\left\{ |a_0|+|a_{n-1}|,1+ \dfrac{M}{|a_0|} \right\}=|a_0|+|a_{n-1}|.$$
	Further, $$N(\F)=\min\left\{ \max\left\{ |a_0|,1+M \right\}, \max\left\{ 1,|a_0|+|a_1|+\ldots +|a_{n-1}|\right\} \right\} =\max\left\{ |a_0|,1+M \right\}.$$
	Suppose $N(\F)=|a_0|$. Then $|a_{n-1}|+1\leq 1+M\leq |a_0|$. Thus $\|W\|_\infty < 2 |a_0|$. And hence
	$\|W\|_\infty < 2 N(\F)$.
	
	Suppose $N(\F)=1+M$. Then  $|a_0|\leq 1+M$. Further $|a_{n-1}|\leq M<1+M$. Thus
	$|a_0|+|a_{n-1}|< 2(1+M)$. Therefore $\|W\|_\infty< 2(1+M)=2N(\F)$.
\end{proof}

In this section we compared upper bounds using the inverse of certain unit sparse companion matrices corresponding to the monic reversal polynomial.
Future work could explore other cases, as we have only explored the case when $\mathbf{u}=\mathbf{0}$
in line (\ref{eC2}). In this case, as was observed in~\cite{DDP}, significant improvement can be made compared to the Frobenius bounds. In particular, we 
observed that $\|W\|_\infty$ can give a better bound than $N(\F)$, and that when $N(\F)< \|W\|_\infty$, it will not
be an improvement by more than a factor of two.

  Further work could be done on exploring the lower bounds using the inverse of an intercyclic companion matrix. As with the upper bounds, 
there will be many options to consider, given the structure of the inverse companion matrix in line (\ref{eC2}).
 De Ter\'{a}n, Dopico, and P\'{e}rez \cite[Example 5.8]{DDP} give an example to illustrate that the lower bound using a Fiedler companion matrix can be significantly better than using a Frobenius companion matrix.

\section{Concluding comments}

It should be noted that there is a gap in the conditions of Theorem~\ref{bestEasilyInverted}, namely the case when $|a_0| > 1$ and $|a_{n-1}|<M$. Under these conditions, we suspect that there is no single matrix whose 1-norm provides a sharper bound than those provided by the 1-norms of other matrices of type $E_c(p^\sharp)^{-1}$. However, the following matrices of type $E_{n-2}(p^\sharp)^{-1}$ are useful. 
In particular, if $1\leq b\leq n-2$, define 
\begin{equation}\label{secondFiedler}
X_{b} =
\left[ \renewcommand{\arraystretch}{1.3}
\begin{array}{ccccc|c|c}
-a_{n-1} & \cdots & -a_{b+1} & \multicolumn{2}{c|}{O} & 0 & -a_0 \\ \hline
\multicolumn{2}{c}{} & & \multicolumn{2}{c|}{} & & \\
\multicolumn{2}{c}{}  & I_{n-2} & \multicolumn{2}{c|}{}  & O & O \\
\multicolumn{2}{c}{} & & \multicolumn{2}{c|}{} & & \\ \hline
\multicolumn{2}{c}{O} & {a_b}/{a_0} &\cdots & {a_1}/{a_0} & 1 & 0 \\
\end{array} \right].
\end{equation}

Note that $X_b$ is the inverse of the matrix 
\begin{equation}
\phantom{X_b=}
\left[ \renewcommand{\arraystretch}{1.3}
\begin{array}{cccccc|c}
\multicolumn{6}{c|}{}&\\
\multicolumn{2}{c}{}  & I_{n-2} & \multicolumn{2}{c}{}  &\multicolumn{1}{c|}{}  & O \\
\multicolumn{6}{c|}{}&\\ \hline
&O& & -\frac{a_b}{a_0}&\cdots&-\frac{a_1}{a_0}&1\\
-\frac{1}{a_0}&-\frac{a_{n-1}}{a_0}&\cdots&-\frac{a_{b+1}}{a_0}&\multicolumn{2}{c|}{O}&0\\
\end{array} \right],
\end{equation}
which is a Fiedler matrix of the monic reversal polynomial $p^\sharp$. 
In the following result, we consider a particular $X_\beta$ whose $1$-norm is derived from the first column, or one of the
last $\beta+1$ columns. 

\begin{theorem}\label{secondFiedlerIsBest} Let $M=\max\{ |a_k|\ \vert\ 1\leq k\leq n-1\}$
and $1\leq c\leq n-2$.
Suppose $1<|a_0|$, $|a_{n-1}|<M$ and $\|X_\beta\|_1 \leq \|X_b\|_1$ for all $b$, $1\leq b\leq n-2$. If \[ \|X_\beta\|_1 \in \left\{ 1+|a_{n-1}|,1+\left|\frac{a_{\beta-1}}{a_0}\right|,\ldots,1+\left|\frac{a_1}{a_0}\right|,|a_0| \right\}, \]
then
$\|X_\beta\|_1 \leq \|A\|_1$ for every matrix $A$ of type $E_c(p^\sharp)^{-1}$.
\end{theorem}
\begin{proof}
Suppose $\|X_\beta\|_1 = 1+|a_{n-1}|$. Observe that $\|A\|_1 \geq 1+|a_{n-1}|$ as $a_{n-1}$ must be in the location $(1,1)$.

Suppose $\|X_\beta\|_1 = |a_0|$. Observe that $\|A\|_1 \geq |a_0|$ as $a_0$ must be in the location $(1,n)$.

Suppose $\|X_\beta\|_1 = 1+\left|\frac{a_{k}}{a_0}\right|$ where $1\leq k\leq \beta-1$. Observe that $\|A\|_1 \geq 1+\left|\frac{a_{k}}{a_0}\right|$ or $\|A\|_1 \geq 1+\left|a_{k}\right| > 1+\left|\frac{a_{k}}{a_0}\right|.$
\end{proof}

\begin{theorem}\label{secondFiedlerIff}
Given $M=\max\{ |a_k|\ \vert\ 1\leq k\leq n-1\}$ and $m =\max\{ k \ \vert\  M=|a_k| \}$.
Suppose $b \leq n-2$. Then
$\|X_b\|_1 < N(\F)$ if and only if 
$1 < |a_0| < 1+M$ and,
for some $b \geq m$, $|a_{b+1}| + \left|\dfrac{a_b}{a_0}\right| < M.$
\end{theorem}
\begin{proof} Suppose $b \leq n-2$.
\[\|X_b\|_1 = \max \left\{ |a_0|,1+|a_{n-1}|,\ldots,1+|a_{b+2}|,1+|a_{b+1}|+\left|\dfrac{a_b}{a_0}\right|,1+\left|\dfrac{a_{b-1}}{a_0}\right|,\ldots,1+\left|\dfrac{a_1}{a_0}\right| \right\}.\]

Suppose $\|X_b\|_1 < N(\F)$.

Suppose $1+M \leq |a_0|$. Then $1<|a_0|+|a_1|+\cdots +|a_{n-1}|$ and $N(\F) = |a_0|$. Since $\|X_b\|_1 \geq |a_0| = N(\F)$, this a contradiction. Therefore $|a_0| < 1+M$. This also implies that $N(\F) \leq 1+M$.

Suppose $|a_0| \leq 1$. Then $\left|\dfrac{a_i}{a_0}\right| \geq |a_i|$ for all $i$, so that 
\[ \|X_b\|_1 \geq 1+M \geq N(\F) \]
which is a contradiction. 
Therefore $|a_0|>1$.

With these two conditions, we have $N(\F) = 1+M$.  
Since $\|X_b\|_1 < N(\F)$, it follows that $m \leq b$.

For the converse, suppose $|a_{b+1}|+\left|\dfrac{a_b}{a_0}\right| < M$ for some $b \geq m$ and $1< |a_0| < 1+M$.
The latter condition implies $N(\F) = 1+M$.
Since $m \leq b$, $1+|a_i| < 1+M$ for all $i \geq b+2$. Further,
$1+\left| \dfrac{a_i}{a_0} \right| < 1+M$ for all $i \leq b-1$ since $|a_0|>1$.
It follows that 
$\|X_b\|_1 < 1+M=N(\F)$.
\end{proof}

\begin{example}\label{xw}
Consider the polynomial \[p = x^8 +4x^7 -x^6 +x^5 +17x^4 +20x^3 -10x^2 -5x +5. \] Observe that $|a_0| = 5 > 1$ and $|a_{n-1}| = 4 < 20 = M$, so we know that Theorem $\ref{bestEasilyInverted}$ part $\ref{superFiedlerIsBest}$ applies. 

One can check that $\|X_5\|_1=|a_0|$ and 
$\|X_5\|_1 \leq \|X_b\|_1$ for all $b$, $1\leq b\leq n-2$. 
Hence, by Theorem $\ref{secondFiedlerIsBest}$
$\|X_5\|_1 \leq \|A\|_1$ for every matrix $A$ of type $E_c(p^\sharp)^{-1}$.
Further, since
$1 < |a_0|=5 < 1+M-|a_0|=16$ the condition in Theorem $\ref{superFiedlerIff}$
is satisfied. And 
since 
$|a_0| = 5 < 21 = 1+M$ and $|a_6|+\left|\frac{a_5}{a_0}\right| = 1.2 < 20 = M$ for $b=5\geq 4$,
the conditions of 
Theorem $\ref{secondFiedlerIff}$ are satisfied.
Therefore, we should expect that $W$ and $X_5$ provide sharper bounds than all other inverse companion matrices in 
$(\ref{eC2})$ with $\mathbf{u}=\mathbf{0}$, including the Frobenius companion matrix. We can verify this directly: 
$N(\F) = \min\left\{\max \left\{ 5,21 \right\}, \max\left\{ 1,63 \right\} \right\} =21,$
\[
W = \left[ \begin{array}{rrrrrrrr}
-4   & 0 & 0 & 0 & 0 & 0 & 0 & -5 \\
 1   & 0 & 0 & 0 & 0 & 0 & 0 &  0 \\
-1   & 1 & 0 & 0 & 0 & 0 & 0 &  0 \\
-2   & 0 & 1 & 0 & 0 & 0 & 0 &  0 \\
 4   & 0 & 0 & 1 & 0 & 0 & 0 &  0 \\
 3.4 & 0 & 0 & 0 & 1 & 0 & 0 &  0 \\
 0.2 & 0 & 0 & 0 & 0 & 1 & 0 &  0 \\
-0.2 & 0 & 0 & 0 & 0 & 0 & 1 &  0 \\
\end{array} \right] %\\
{\rm \quad and \quad}
%\|W\|_{\infty} &= 9 \\
X_5 = \left[ \begin{array}{rrrrrrrr}
-4 & 1   & 0   & 0 &  0 &  0 & 0 & -5 \\
 1 & 0   & 0   & 0 &  0 &  0 & 0 &  0 \\
 0 & 1   & 0   & 0 &  0 &  0 & 0 &  0 \\
 0 & 0   & 1   & 0 &  0 &  0 & 0 &  0 \\
 0 & 0   & 0   & 1 &  0 &  0 & 0 &  0 \\
 0 & 0   & 0   & 0 &  1 &  0 & 0 &  0 \\
 0 & 0   & 0   & 0 &  0 &  1 & 0 &  0 \\
 0 & 0.2 & 3.4 & 4 & -2 & -1 & 1 &  0 \\
\end{array} \right].
%\|X_5\|_1 &= 5
%\end{align*}
\]
Both $W$ and $X_5$ provide sharper bounds than $N(\F)$, and the ratios are greater than two:
\[
\frac{N(\F)}{\|W\|_\infty} = \frac{21}{9} \doteq 2.3 
\rm{ \quad and \quad}
\frac{N(\F)}{\|X_5\|_1} = \frac{21}{5} = 4.2 
\]
In this example, $\|X_5\|_1 < \|W\|_\infty$ but this may not always be the case.
\end{example}

\begin{remark}\label{large} Example~$\ref{xw}$ illustrates that ${\|W\|_\infty}$ and ${\|X_b\|_1}$ can be significantly smaller
than $N(\F)$. In fact, the ratios can be made arbitrarily large. In particular, 
De Ter\'{a}n, Dopico, and P\'{e}rez \cite[Example 5.6]{DDP} provide an example with
${\|X_1\|_1}={\|W\|_\infty}=1+10^m$ and $N(F)=1+10^{2m}$ for any integer $m>0$. 
\end{remark}

Since $X_b$ is the inverse of a Fiedler matrix of a monic reversal polynomial, Theorem~5.4(b) of \cite{DDP} implies
that if $\|X_b\|_1 < \|W\|_\infty$ for a given polynomial $p$, then $\|W\|_\infty$ is at most twice $\|X_b\|_1$. The following theorem is a slight refinement for $X_b$.

\begin{theorem}\label{superJogRel} Let $p$ be a polynomial as in (\ref{poly}). If $1\leq b\leq n-2$ and
$|a_0|> 1$, then
\[ \frac{\|W\|_{\infty}}{\|X_b\|_1} \leq 2-\frac{1}{|a_0|}. \]
\end{theorem}
\begin{proof}
Suppose $\|X_b\|_1 < \|W\|_\infty$. Then $\|W\|_\infty = |a_0| + |a_{n-1}|$ since otherwise $\|W\|_\infty = 1+M/|a_0|$ but $\|X_b\|_1 \geq 1+M/|a_0|$. 
Let $U=\max\{1+M/|a_0|, 1+|a_{n-1}|, |a_0|\}$. Observe that $\|X_b\|_1\geq U$.

Suppose that $U=|a_0|.$ 
Then \[ \frac{\|W\|_\infty}{\|X_b\|_1} \leq \frac{|a_0|+|a_{n-1}|}{|a_0|} \leq \frac{|a_0|+|a_0|-1}{|a_0|} = 2 - \frac{1}{|a_0|}. \]
Suppose that $U=1+|a_{n-1}|$.  
Then \[ \frac{\|W\|_\infty}{\|X_b\|_1}\leq 
\frac{|a_0|+|a_{n-1}|}{1+|a_{n-1}|} = 1+ \frac{|a_0|-1}{1+|a_{n-1}|} \leq 1+\frac{|a_0|-1}{|a_0|} = 2 - \frac{1}{|a_0|}. \]
Suppose that $U=1+M/|a_0|$. 
If $|a_0| \geq 1+|a_{n-1}|$ then
\[ \frac{\|W\|_\infty}{\|X_b\|_1} \leq 
\frac{|a_0|+|a_{n-1}|}{1+M/|a_0|} \leq \frac{|a_0|+|a_{n-1}|}{|a_0|} \leq 2 - \frac{1}{|a_0|}. \] 
Otherwise, $|a_0| < 1+|a_{n-1}|$, and \[ \frac{\|W\|_\infty}{\|X_b\|_1} \leq \frac{|a_0|+|a_{n-1}|}{1+M/|a_0|} \leq 
\frac{|a_0|+|a_{n-1}|}{1+|a_{n-1}|} = 1+ \frac{|a_0|-1}{1+|a_{n-1}|} < 1+\frac{|a_0|-1}{|a_0|} = 2 - \frac{1}{|a_0|}. \]
\end{proof}

As a final note, we would like to point out that there is close connection between the unit sparse companion matrices
and the set of block Kronecker pencils described in \cite{DLPV}. In particular, if $C(p)$ is a unit sparse companion matrix associated
with a monic polynomial $p(x)$, then the matrix pencil $\lambda I_n-C(p)$ is of the form
$$\lambda I_n-C(p)=
\left[\begin{array}{c|c}
\mathcal{L}_m(\lambda ) &O\\ \hline
K(\lambda ) & \mathcal{L}_{n-m-1}(\lambda)^T
\end{array}\right],$$
with $\mathcal{L}_k(\lambda )$ a $k \times (k+1)$ matrix of the form
$$\mathcal{L}_k(\lambda )=\left[ \begin{array}{ccccr}
\lambda  &-1&0&\cdots&0\\
0&\lambda &-1&\ddots&\vdots\\
\vdots&\ddots&\ddots&\ddots&0\\
0&\cdots&0&\lambda &-1
\end{array}\right].
$$
This matrix pencil is permutationally similar to a block Kronecker pencil. The connection may
provide opportunities to extend this work to (monic) matrix polynomials.

\textbf{Acknowledgement.} {Research supported in part by NSERC Discovery Grant 203336 and an NSERC USRA.
We thank a referee for a very careful reading of our paper and for pointing out the connection to \cite{DLPV} 
in the conclusion.}

\end{document}